\documentclass[11pt]{amsart}
\usepackage{amsmath,amssymb,amsbsy,amsfonts,amsthm,latexsym,
            amsopn,amstext,amsxtra,euscript,amscd,color}

\PassOptionsToPackage{hyphens}{url}\usepackage{hyperref}
 
 \usepackage[capbesideposition=outside,capbesidesep=quad]{floatrow}

\restylefloat{table}
\restylefloat{table}
            
\usepackage{multirow,caption}
            
\usepackage{amscd}
\usepackage{color,enumerate}

\usepackage[colorinlistoftodos,prependcaption,textsize=tiny]{todonotes}

\newtheorem{thm}{Theorem}[section]
\newtheorem{lem}[thm]{Lemma}
\newtheorem{cor}[thm]{Corollary}
\newtheorem{prop}[thm]{Proposition}

\newtheorem{conj}[thm]{Conjecture}

\newtheorem{rmk}[thm]{Remark}

\newtheorem{thm-con}[thm]{Theorem-Conjecture}
\numberwithin{equation}{section}

\theoremstyle{definition}
\newtheorem{defn}[thm]{Definition}

\newcommand{\f}{\Bbb F}
\newcommand{\F}{\Bbb F}
\newcommand{\cW}{\mathcal{W}}
\newcommand{\Ker}{{\rm Ker}}
\newcommand{\im}{{\rm Im}}
\newcommand{\cL}{{\mathcal L}}
\newcommand{\cA}{{\mathcal A}}

\newcommand{\Tr}{{\rm Tr}}

\newcommand*\colvec[3][]{
    \begin{pmatrix}\ifx\relax#1\relax\else#1\\\fi#2\\#3\end{pmatrix}
}

\begin{document}

\title[On the $c$-differential uniformity of certain maps over finite fields]{On the $c$-differential uniformity of certain maps over finite fields} 

\author[S. U. Hasan]{Sartaj Ul Hasan}
\address{Department of Mathematics, Indian Institute of Technology Jammu, Jammu 181221, India}
\email{sartaj.hasan@iitjammu.ac.in}

\author[M. Pal]{Mohit Pal}
\address{Department of Mathematics, Indian Institute of Technology Jammu, Jammu 181221, India}
\email{2018RMA0021@iitjammu.ac.in}

\author[C.~Riera]{Constanza Riera}
\address{Department of Computer Science,
 Electrical Engineering and Mathematical Sciences,   Western Norway University of Applied Sciences,  5020 Bergen, Norway}
 \email{csr@hvl.no}
 
 \author[P.~St\u anic\u a]{Pantelimon~St\u anic\u a}
 \address{Applied Mathematics Department, Naval Postgraduate School, Monterey 93943, USA}
\email{pstanica@nps.edu}

\begin{abstract}
We give some classes of power maps with low $c$-differential uniformity over finite fields of odd characteristic, {for $c=-1$}.  Moreover, we give a necessary and sufficient condition for a linearized polynomial to be a perfect $c$-nonlinear function and investigate conditions when perturbations of perfect $c$-nonlinear (or not) function via an arbitrary Boolean or $p$-ary function is perfect $c$-nonlinear. In the process, we obtain a class of polynomials that are perfect $c$-nonlinear  for all $c\neq 1$, in every characteristic. The affine, extended affine and CCZ-equivalence is also looked at, as it relates to $c$-differential uniformity.
\end{abstract}



\maketitle

\section{Introduction}\label{S1}
Let $p$ be an odd prime (unless stated otherwise), and let $n$ be a positive integer. We denote by $\mathbb{F}_q$ the finite field with $q=p^n$ elements and by $\mathbb{F}_{q}^*$, the multiplicative cyclic group of nonzero elements of $\mathbb{F}_q$. 
We call a function from $\F_{p^n}$ to $\F_p$  a {\em $p$-ary  function} on $n$ variables.
%

For positive integers $n$ and $m$, any map $F:\F_{p^n}\to\F_{p^m}$ is called a {\em vectorial $p$-ary  function}, or {\em $(n,m)$-function}. When $m=n$, $F$ can be uniquely represented as a univariate polynomial over $\F_{p^n}$ of the form
$
F(x)=\sum_{i=0}^{p^n-1} a_i x^i$,\ $a_i\in\F_{p^n}.
$

The Walsh transform $ \cW_{F}(a,b)$ of an $(n,m)$-function $F$ at $a\in \F_{p^n}, b\in \F_{p^m}$ is 
defined as
\[
  \cW_{F}(a,b)=\sum_{x\in\F_{p^n}} \zeta_p^{\Tr_m(bF(x))-\Tr_n(ax)},
\]
where $\zeta= e^{\frac{2\pi i}{p}}$ is a $p$-root of unity and $\Tr:\f_{p^n}\to \f_p$ is the absolute trace function, given by $\displaystyle \Tr(x)=\sum_{i=0}^{n-1} x^{p^i}$.

We also say that $\alpha\in\F_{p^n}^*$ is a {\em $\beta$-linear structure} for $F$, if $F(x+\alpha)-F(x)=\beta$, for all $x\in\F_{p^n}$.

Let $F$ be a function $F: \mathbb{F}_{p^n} \rightarrow \mathbb{F}_{p^n}$. For any $a,b \in \mathbb{F}_{p^n}$, we denote by $N(a,b)$ the number of solutions $x\in \mathbb{F}_{p^n}$ to $F(x+a)-F(x)=b$. Let 
$$
\Delta_F = \mbox{max}\{N(a,b)~|~a,b \in \mathbb{F}_{p^n}, a\neq0 \},
$$ 
then a function $F$ is called {\em differentially $\delta$-uniform} if $\Delta_F= \delta$. When $\Delta_F=1$, we call $F$ to be {\em perfect nonlinear} (PN) function. In the literature, the PN functions sometimes also referred to as {\em planar functions}. It is straightforward to see that in the case of characteristic two, $x+a$ and $x$ have the same image and hence there is no PN function over a finite field of even characteristic.

Deviating from the usual differentials $(F(x + a), F(x))$, Borisov et. al. ~\cite{Bori} introduced the notion of so called multiplicative differentials of the form $(F(cx), F(x))$ and they used this new type of differentials to attack some existing ciphers. Motivated by the multiplicative differential as discussed in \cite{Bori}, two of us, along with Ellingsen, Felke,  and Tkachenko~\cite{CDU} defined a new (output) multiplicative differential in the following way. 
\begin{defn}
\label{defCDU}
Let $F$ be a function from a finite field $\mathbb{F}_{p^n}$ to itself. For any $a,c \in \mathbb{F}_{p^n}$,  the (multiplicative) $c$-derivative of $F$ with respect to $a$ is defined as 
$$_cD_a F(x)= F(x+a)-cF(x)~\mbox{for all}~x \in \mathbb{F}_{p^n}.$$
For $a,b\in \mathbb{F}_{p^n}$, let $_c\Delta_F(a,b) = \#\{x\in \mathbb{F}_{p^n}: F(x+a)-cF(x)=b \}.$ The {\em $c$-differential uniformity} of  $F$, denoted as $_c\Delta_F$,  is then defined  as $$_c\Delta_F := \mbox{max}\{_c\Delta_F(a,b): a,b \in \mathbb{F}_{p^n},~ \mbox{and}~ a\neq0 ~ \mbox{if}~ c=1 \}.$$ When $_c\Delta_F = \delta$, we say that {\em $c$-differential uniformity} of $F$ is $\delta$.
\end{defn}

It is easy to see that when $c=1$, $c$-differential uniformity coincides with the usual notion of differential uniformity.
If $\delta=1$ then $F$ is called {\em perfect $c$-nonlinear} (PcN) function and when $\delta=2$ then $F$ is called {\em almost perfect $c$-nonlinear} (APcN) function. 

Recently, Bartoli and Timpanella~\cite{BT}   gave a generalization of planar functions as follows.
\begin{defn}\label{defBT}
Let $\beta \in \mathbb{F}_{p^n} \backslash \{0,1\}$. A function $F : \mathbb{F}_{p^n} \rightarrow \mathbb{F}_{p^n}$ is a $\beta$-planar  function in $\mathbb{F}_{p^n}$ if $\forall~ \gamma \in \mathbb{F}_{p^n},~~~~ F(x+\gamma) - \beta F(x)$ is a permutation of $\mathbb{F}_{p^n}.$
\end{defn}
In the particular case, when $\beta =-1$, then $\beta$-planar function is called quasi-planar. In view of the Definition~\ref{defCDU}, the $\beta$-planar functions are simply  PcN functions and quasi-planar functions are PcN functions with $c=-1$. In what follows, we shall adopt the notations given in Definition~\ref{defCDU}. It is easy to see from the definition of PcN function that when $c\neq1$ and $a=0$ then $F(x+a)-cF(x)=(1-c)F(x)$ is a permutation polynomial if and only if $F(x)$ is a permutation polynomial. Therefore, we shall consider the perfect $c$-nonlinearity of permutation polynomials only.

In the next section, we establish a relation between the difference function of the power map $x^d$ and Dickson polynomial of first kind over $\mathbb F_{p^n}$, for $c=-1$. In fact, such a relationship has its origin in \cite[Proposition 8]{Li}, where it was established for the fields of characteristic 3. However, it turns out that the sufficient conditions of \cite[Proposition 8]{Li} are not correct. Bartoli and Timpanella  in their recent work \cite[Theorem 6.1]{BT} extended and corrected \cite[Proposition 8]{Li}.

Now we give the structure of the paper. In Section~\ref{S2}, we show that over a finite field $\mathbb{F}_{p^n}$ of odd characteristic, the conditions of  \cite[Theorem 6.1]{BT} are also sufficient. As a consequence, we shall show that $\displaystyle x^{\frac{p^\ell+1}{2}}$ is PcN for $c=-1$ over $\mathbb{F}_{p^n}$ if and only if $\ell = 0$ or $\displaystyle \frac{\ell}{\gcd~(\ell,n)}$ is even (see also~\cite{ZM}, or~\cite{RS20}, where this function was thoroughly analyzed). In Section~\ref{S3}, we give four classes of power maps whose $c$-differential uniformity for $c=-1$ is $2, 3, 6~ \mbox{and}~7$. In Section~\ref{S4}, we give all values of $d$ for which $x^d$ is PcN over the finite fields $\mathbb{F}_{3^5}$, $\mathbb{F}_{5^5}$ and $\mathbb{F}_{7^5}$, respectively, for $c=-1$. Following the pattern of the computational results, we propose a conjecture about the plausible values of $d$ for which $x^d$ is PcN over $\mathbb{F}_{p^5}$ for $c=-1$. Similarly in Section~\ref{S5}, we give all values of $d$ for which $x^d$ is PcN over the finite fields $\mathbb{F}_{3^7}$, $\mathbb{F}_{5^7}$ and $\mathbb{F}_{7^7}$, respectively, for $c=-1$. Following the pattern of the computational results, we propose another conjecture about the plausible values of $d$ for which $x^d$ is PcN over $\mathbb{F}_{p^7}$ for $c=-1$. In Section~\ref{S6}, for $c\neq 1$, we give a necessary and sufficient condition for a linearized polynomial to be PcN. We also find necessary and sufficient conditions for the sum $F+\gamma f$ to be PcN, where $\gamma\in\F_{p^n}$, $F$ is PcN and $f$ is any Boolean function. We also show that in some instances such perturbations do not produce PcN functions. We further discuss the affine, extended affine and CCZ-equivalence as it relates to $c$-differential uniformity. Finally, in Section~\ref{S7} we present the conclusion of the paper.

\section{P$c$N power maps and Dickson polynomials}
 \label{S2}
Recall that for $c=-1$, a polynomial function $F(x)$ is called PcN over $\mathbb{F}_{p^n}$ if the corresponding mapping $x \rightarrow F(x+a)+F(x)$ is a permutation of $\mathbb{F}_{p^n}$ for each $a \in \mathbb F_{p^n}$. Therefore, a power map $x^d$ is PcN if and only if $(x+a)^d+x^d$ is a permutation of $\mathbb{F}_{p^n}$ for each $a \in \mathbb{F}_{p^n}$. Now we present some lemmas that will be useful in the sequel. Throughout this section, we shall assume that $c=-1$, whenever we refer to PcN functions.

\begin{lem}
\label{equiv-def}
A monomial $x^d$ is perfect $(-1)$-nonlinear in $\mathbb{F}_{p^n}$ if and only if  $x^d$ and $(x+1)^d+(x-1)^d$ are permutations of $\mathbb{F}_{p^n}$. 
\end{lem}
\begin{proof}
Let $F(x)=x^d$; then, by definition, $F$ is a PcN function if and only if $(x+a)^d+x^d$ is a permutation of $\mathbb{F}_{p^n}$ for all $a\in\mathbb{F}_{p^n}$. For $a = 0$, we have $(x+a)^d+x^d=2x^d$, and $2x^d$  is clearly a permutation of $\mathbb{F}_{p^n}$ if and only if $x^d$ is a permutation of $\mathbb{F}_{p^n}$.
For $a \neq 0$, we have
\allowdisplaybreaks
\begin{align*}
&(x+a)^d+x^d ~~~\mbox{is a permutation of $\mathbb{F}_{p^n}$} \\
\Longleftrightarrow \quad  &a^d \left[\left(\frac{x}{a}+1\right)^d + \left(\frac{x}{a}\right)^d\right] \mbox{ is a permutation of $\mathbb{F}_{p^n}$} \\
\Longleftrightarrow \quad& \left(\frac{x}{a}+1\right)^d + \left(\frac{x}{a}\right)^d \mbox{ is a permutation of $\mathbb{F}_{p^n}$} \\
\Longleftrightarrow \quad & \left(y+1\right)^d + y^d \mbox{ is a permutation of $\mathbb{F}_{p^n}$};~ \mbox{where}~ a y=x \\
\Longleftrightarrow \quad & \left(\frac{2y+1+1}{2}\right)^d + \left(\frac{2y+1-1}{2}\right)^d \mbox{ is a permutation of $\mathbb{F}_{p^n}$} \\
\stackrel{z:=2y+1}{\Longleftrightarrow} \quad & \left(\frac{1}{2}\right)^d\left[\left(z+1\right)^d + \left(z-1\right)^d\right] \mbox{ is a permutation of $\mathbb{F}_{p^n}$} \\
 \Longleftrightarrow \quad & \left(z+1\right)^d + \left(z-1\right)^d \mbox{ is a permutation of $\mathbb{F}_{p^n}$}.
 \end{align*}
This completes the proof of the lemma.
\end{proof}
In what follows, we shall adopt this definition of PcN function for power maps, when $c=-1$.  One of the motivations behind considering this definition is that we can establish a connection between $(x+1)^d+(x-1)^d$ and $d$-th Dickson polynomial of the first kind. We recall the Dickson's original approach of defining the Dickson polynomial $D_d(x,a)$, which was essentially based on the relationship between the sum of $d$-th powers and elementary symmetric functions. In fact, the $d$-th Dickson polynomial of the first kind $D_d(x,a) \in \mathbb F_q[x]$ ($q$ a power of the prime $p$) admits the following representation
\begin{equation}
\begin{split}
\label{dickson}
 u_1^d+u_2^d&= \sum_{i=0}^{\lfloor \frac{d}{2} \rfloor} \frac{d}{d-i} \binom{d-i}{ i} (-u_1u_2)^i (u_1+u_2)^{d-2i} \\
 &= D_d(u_1+u_2, u_1u_2),
\end{split}
\end{equation}
where $u_1,u_2$ are indeterminates and $\displaystyle D_d(x,a)= \sum_{i=0}^{\lfloor \frac{d}{2} \rfloor} \frac{d}{d-i}  \binom{d-i}{ i}(-a)^i x^{d-2i}$.

We will be using in some places Hilbert's Theorem 90 (see~\cite{Bo90}), which states that if $\mathbb{F}\hookrightarrow \mathbb{K}$  is a cyclic Galois extension and $\sigma$ is a generator of the Galois group ${\rm Gal}(\mathbb{K}/\mathbb{F})$, then  the relative trace $\displaystyle \Tr_{\mathbb{K}/\mathbb{F}}(x)=\sum_{i=0}^{|{\rm Gal}(\mathbb{K}/\mathbb{F})|-1}\sigma^i(x)=0$, $x\in \mathbb{K}$, if and only if $x=\sigma(y)-y$, for some $y\in\mathbb{K}$.

We now recall a result of N{\"o}bauer \cite{WN}, which we shall often use, regarding the permutation behavior of  Dickson polynomial of the first kind over the finite field $\mathbb{F}_{p^n}$. 
\begin{lem}\cite{WN}
\label{DPP}
Let $a\in\mathbb{F}_{p^n}^*$. The $d$-th Dickson polynomial of the first kind $D_d(x,a)$ permutes the elements of finite field $\mathbb{F}_{p^n}$ if and only if $\displaystyle \gcd~(d, p^{2n}-1)=1.$ 
\end{lem}
The following lemma  will be used throughout.
\begin{lem}\cite[Lemma 9]{CDU} 
\label{gcd}
 Let $p$ be a prime number and $\ell,n$ be positive integers such that  $\ell \leq n$. Then:
 \begin{enumerate}
  \item[$(1)$] If $p$ is odd, then $\gcd(p^\ell+1, p^n-1)=2$ if $\displaystyle \frac{n}{\gcd(\ell,n)}$ is odd.
  \item[$(2)$] If $p$ is odd, then  $\gcd(p^\ell+1, p^n-1)=\displaystyle p^{\gcd(\ell,n)}+1$ if $\displaystyle \frac{n}{\gcd(\ell,n)}$ is even.
   \item[$(3)$] If $p=2$, then  $\gcd(2^\ell+1,2^n-1)=\displaystyle \frac{2^{\gcd(n,2\ell)}-1}{2^{\gcd(n,\ell)}-1}$.
 \end{enumerate}
\end{lem}

The following lemma gives a nice connection between the difference function of the power map $x^d$ and the Dickson polynomial for first kind over $\mathbb F_{3^n}$, for $c=-1$.
\begin{lem}
\cite[Proposition 8]{Li} 
\label{Intro-Li}
For a positive odd integer $n$ with $n\geq3$, if $d\equiv -1 \pmod 3$ and $\mbox{gcd}~(d, 3^{2n}-1) = 1$, then
\begin{equation}\label{rel}
 (x+1)^d+(x-1)^d = 2D_d(x,1)
\end{equation}
is a permutation of $\F_{3^n}$, where $D_d(x, 1)$ is the Dickson polynomial of the first kind.
\end{lem}
As alluded to in Introduction, the sufficient conditions in the above lemma do not hold, and the counterexamples can be found using easy computer searches. For instance, when $n=5$ and $d=17$, $d$ clearly satisfies the conditions of Lemma~\ref{Intro-Li}, but $(x+1)^{17}+(x-1)^{17} \neq 2D_d(x,1)$. Bartoli and Timpanella~\cite[Theorem 6.1]{BT} provided the correct conditions on $d$ for which (\ref{rel}) holds over finite fields of odd characteristic. However, it appears that there is a missing case $(k=0)$ in~\cite[Theorem 6.1]{BT}, which we shall include here. The following theorem provides a relationship between the difference function of the power map $x^d$ and the Dickson polynomial of first kind over $\mathbb F_{p^n}$, for $c =-1$. 

\begin{thm}
\label{maint}
Let $p$ be an odd prime, $d$ be a positive integer such that $d=a_0+a_1p+a_2p^2+ \cdots +a_kp^k$ for some $k\geq 0$, where $a_i \in \{0, 1, \cdots, p-1 \}$ and $a_0, a_k \neq 0$, then  $\displaystyle (x+1)^d+(x-1)^d=2D_d(x,\epsilon)$ for some $\epsilon \in \mathbb{F}_p^*$ if and only if either
\begin{enumerate}
 \item[$(1)$] $d= 1,2,3;$ or
 \item[$(2)$] $\displaystyle a_0 =\frac{p+1}{2}$ and $\displaystyle a_j = \frac{p-1}{2}$  $\forall j \in \{1, 2, \ldots ,k\}$ $\left(\text{thus}, \displaystyle d=\frac{p^{k+1}+1}{2}\right)$.
\end{enumerate}
\end{thm}
\begin{proof}
The necessity of the theorem has already been proved in~\cite{BT} for all $k$ except for the case $k=0$. Here we shall prove the necessity for the case $k=0$. In this case, we have $d=a_0 \in \{1, \ldots, p-1 \}$. We now consider two cases, namely, $p=3$ and $p>3$. If $p=3$, the only possible values for $d$ are 1 and 2 and we are done. If $p>3$ (hence, we can assume  $d \geq 4$, since the values $d=1,2,3$ were already covered in Condition~$(1)$), we shall show that the only possible value of $a_0$ is $\displaystyle \frac{p+1}{2}$. It is given that 
$$\displaystyle (x+1)^d+(x-1)^d=2D_d(x,\epsilon)$$ 
for some $\epsilon \in \mathbb{F}_p^*$. By using binomial expansion on the left in the above equation, and by comparing the coefficients on both sides, we have

$$ 
\binom{d}{ 2i} \equiv \frac{d}{d-i}  \binom{d-i}{ i}(-\epsilon)^i  \pmod p,
$$
for all $i\in \left\{0, 1, \ldots, \lfloor \frac{d}{2} \rfloor \right\}$. 

Surely, for $i=0$, the previous claim is obviously true. For $i = 1$, we have
$$ 
 \frac{a_0(a_0-1)}{2} \equiv -\epsilon \cdot a_0   \pmod p,
$$
which is true if and only if $\displaystyle \epsilon \equiv \frac{1-a_0}{2}   \pmod p.$ 

For $i = 2$, we have 
\begin{equation} \label{eqi2}
 \frac{a_0(a_0-1)(a_0-2)(a_0-3)}{24} \equiv \frac{a_0(a_0-1)^2(a_0-3)}{8}  \pmod p.
\end{equation}
Now since $a_0 \in \{4,\ldots, p-1\}$, the congruence \eqref{eqi2} reduces to $2a_0\equiv 1 \pmod p$ which is true if and only if $\displaystyle a_0 = \frac{p+1}{2}$. Therefore for $k=0$ and $d\geq4$, $\displaystyle \frac{p+1}{2}$ is the only possible value for $a_0$. Hence, the necessity of the theorem for the case $k=0$ is established.

%
%

Next, we shall proceed to prove the sufficiency of the theorem.
When $d=1$, then $\displaystyle (x+1)^d+(x-1)^d=2x=2D_d(x,\epsilon)$ for any $\epsilon \in \mathbb{F}_p^*.$ When $d=2$, then $(x+1)^d+(x-1)^d=2(x^2+1) = \displaystyle 2D_d\left(x, -\frac{1}{2}\right)$. When $d=3$, then $(x+1)^d+(x-1)^d = 2(x^3+3x) = 2D_d(x,-1)$. For $d\geq4$, we shall show that  
 $$(x+1)^d+(x-1)^d=2D_d\left(x,\frac{1}{4}\right).$$
 Since we evaluate Dickson's polynomial over some extension of the involved prime field, $\F_p$, we assume that the variables take values in the extension 
 $\F_q$ of $\F_p$ ($q=p^n$, for some $n$). 
 Now, for $\alpha\in \F_q$, we let $\displaystyle u_1= \frac{u}{2}\in\mathbb{F}_{q^2}$ and $\displaystyle u_2 = \frac{u^{-1}}{2} \in \mathbb{F}_{q^2}$, where $u,u^{-1}$ are the roots of the polynomial $z^2-2\alpha z+1 \in \mathbb{F}_q[z]$. Then, the sum of the roots is $\displaystyle 2\alpha=u+u^{-1} \in \mathbb{F}_q$, and Equation~\eqref{dickson} reduces to
\begin{equation*}
\begin{split}
 D_d\left(\frac{u+u^{-1}}{2},\frac{1}{4}\right) &= \left(\frac{u}{2}\right)^d+ \left(\frac{u^{-1}}{2}\right)^d\\
D_d\left(\alpha ,\frac{1}{4}\right) &= \left(\frac{u}{2}\right)^d+ \left(\frac{u^{-1}}{2}\right)^d.
\end{split} 
\end{equation*}
One may note that when $d=a_0+a_1p+a_2p^2+ \cdots +a_kp^k$ for some $k\geq 0$ and $a_0 =\frac{p+1}{2}$ and $a_j = \frac{p-1}{2}$, for all $j \in \{1, 2, \ldots ,k\}$, then 
\[
d=\frac{p+1}{2}+\frac{p-1}{2}\sum_{j=1}^k p^j= \frac{p+1}{2}+\frac{p-1}{2}\, p\,\frac{p^k-1}{p-1}=\frac{p^{k+1}+1}{2}.
\]
Now, we have (with $\ell=k+1$)
\allowdisplaybreaks
\begin{align*}
(\alpha+1)^{\frac{p^{\ell}+1}{2}}+ (\alpha-1)^{\frac{p^{\ell}+1}{2}}&= \left(\frac{u+u^{-1}}{2}+1\right)^{\frac{p^{\ell}+1}{2}}+ \left(\frac{u+u^{-1}}{2}-1 \right)^{\frac{p^{\ell}+1}{2}} \\
  &=\left(\frac{1}{2}\right)^{\frac{p^{\ell}+1}{2}}\left((u+u^{-1}+2)^{\frac{p^{\ell}+1}{2}}+ (u+u^{-1}-2)^{\frac{p^{\ell}+1}{2}}\right)\\
  &=\left(\frac{1}{2u}\right)^{\frac{p^{\ell}+1}{2}}\left((u^2+2u+1)^{\frac{p^{\ell}+1}{2}}+ (u^2-2u+1)^{\frac{p^{\ell}+1}{2}}\right)\\
  &=\left(\frac{1}{2u}\right)^{\frac{p^{\ell}+1}{2}}\left((u+1)^{p^{\ell}+1}+ (u-1)^{p^{\ell}+1}\right) \\
  &=\left(\frac{1}{2u}\right)^{\frac{p^{\ell}+1}{2}}\left(2u^{p^{\ell}+1}+2 \right) \\ 
  &=2 \left(\frac{1}{2}\right)^{\frac{p^{\ell}+1}{2}}
  \left(u^\frac{{p^{\ell}+1}}{2}+(u^{-1})^\frac{{p^{\ell}+1}}{2}  \right) \\ 
  &=2\left(\left(\frac{u}{2}\right)^\frac{{p^{\ell}+1}}{2}+\left(\frac{u^{-1}}{2}\right)^\frac{{p^{\ell}+1}}{2}\right) \\
  &=2D_{\frac{p^{\ell}+1}{2}}\left(\frac{u+u^{-1}}{2},\frac{1}{4}\right) \\
  &=2D_{\frac{p^{\ell}+1}{2}}\left(\alpha,\frac{1}{4}\right).
 \end{align*} 
Hence, the theorem is proved.
\end{proof}

\begin{rmk}
Theorem~\textup{\ref{maint}} above completes Theorem $6.1$ of~\cite{BT}.  Proposition~$8$ of~\cite{Li} is a particular case of the above theorem with $p=3$. Also, the above theorem provides a simpler proof of~\cite[Proposition 4.1]{BT} in the particular case of $\ell=2$.
\end{rmk}

Our focus is now to study the perfect $c$-nonlinearity of the power map $x^{\frac{p^\ell +1}{2}}$ over $\mathbb{F}_{p^n}$, where $\ell \geq 0$ and $n>1$ (note that this has been also investigated in~\cite{RS20}). As alluded to in Introduction, we shall consider the perfect $c$-nonlinearity of permutation polynomials only. In view of this, we shall first examine the permutation behaviour of the power map $x^{\frac{p^\ell +1}{2}}$. We may impose a restriction of $\ell <n$, so as to ensure that the exponent ${\frac{p^\ell +1}{2}}$ does not exceed $p^n-1$. The following theorem gives the necessary and sufficient conditions on $\ell$ and $n$ for which the power map $x^{\frac{p^\ell +1}{2}}$ is a permutation of $\mathbb{F}_{p^n}$. Surely, we can find it as a particular case of existing permutation classes, but our proof is short enough to warrant an inclusion here.

\begin{thm} 
\label{l2}
The power map $x^{\frac{p^\ell +1}{2}}$ is a permutation of $\mathbb{F}_{p^n}$ if and only if any one of the following conditions hold\textup{:}
 \begin{enumerate}
  \item[$(1)$] $\ell = 0$\textup{;}
   \item[$(2)$] $\ell$ is even and $n$ is odd\textup{;}
   \item[$(3)$] $\ell$ is even and $n$ is even together with $t_2 \geq t_1$, where $n = 2^{t_1}u$ and $\ell = 2^{t_2}v$ such that $2 \nmid u,v$\textup{;}
  \item[$(4)$] $\ell$ is odd, $n$ is odd and $p \equiv 1 \pmod 4$.
 \end{enumerate}
\end{thm}
\begin{proof}
The case $\ell=0$ is trivial. For $\ell \neq 0$, if the exponent $\displaystyle \frac{p^{\ell}+1}{2}$ is even, $\displaystyle \mbox{gcd}~\left(\frac{p^{\ell}+1}{2}, p^n-1\right) \geq 2$ and thus, the power map $\displaystyle X^{\frac{p^{\ell}+1}{2}}$ is not a permutation of $\mathbb F_{p^n}$. We shall, therefore, consider the case when $\displaystyle \frac{p^{\ell}+1}{2}$ is odd. It is easy to see that $\displaystyle \frac{p^{\ell}+1}{2}$ is odd if and only if $\ell$ is even or $\ell$ is odd and $p\equiv 1 \pmod 4.$
If we assume that $\displaystyle \frac{p^{\ell}+1}{2}$ is odd, then a direct application of Lemma~\ref{gcd} shows that $\displaystyle X^{\frac{p^\ell +1}{2}}$ is a permutation of $\mathbb{F}_{p^n}$ if and only if $\displaystyle \gcd~\left(\frac{p^{\ell}+1}{2}, p^n-1\right) = 1$, that is, $ \displaystyle \gcd~\left(p^{\ell}+1, p^n-1\right) = 2$, which is equivalent to $\displaystyle \frac{n}{\gcd(\ell,n)}$ is odd.
Further, under the assumption that $\displaystyle \frac{p^{\ell}+1}{2}$ is odd,  we observe that $\displaystyle \frac{n}{\gcd(\ell,n)}$ is odd if and only if one of later three conditions of the statement of the theorem holds and hence, the theorem is proved.
\end{proof}

Although the map $\displaystyle x^{\frac{p^\ell +1}{2}}$ is a permutation of $\mathbb{F}_{p^n}$ when both $\ell,n$ are odd and $p \equiv 1 \pmod 4$,  the following theorem tells that it ceases to be perfect $(-1)$-nonlinear over $\mathbb{F}_{p^n}$ (compare with~\cite[Theorem 8]{RS20}).

\begin{thm} 
\label{l3}
If both $\ell,n$ are odd and $p\equiv 1 \pmod 4$, then the power map $x^{\frac{p^\ell +1}{2}}$ is not perfect $(-1)$-nonlinear over $\mathbb{F}_{p^n}$.
\end {thm}
\begin{proof}
Since $\ell$ is odd and $p \equiv 1 \pmod 4$, $ \displaystyle \frac{p^{\ell}+1}{2}$ is odd. Now, by a direct application of Lemma~\ref{equiv-def}, Theorem~\ref{maint} and Lemma~\ref{DPP} at the appropriate places, we obtain the following equivalence 
\allowdisplaybreaks
\begin{align*}
& x^{\frac{p^\ell +1}{2}}~\text{is PcN over}~\mathbb F_{p^n} \\
 \iff & \displaystyle (x+1)^{\frac{p^{\ell}+1}{2}}+(x-1)^{\frac{p^{\ell}+1}{2}}~\text{is a permutation of}~ \mathbb{F}_{p^n}\\
 \iff & \displaystyle D_{\frac{p^{\ell}+1}{2}}\left(x,\frac{1}{4}\right)~\text{is a permutation of}~ \mathbb{F}_{p^n}, \forall ~1\leq \ell < n\\ 
\iff &\displaystyle \gcd~\left(\frac{p^{\ell}+1}{2}, p^{2n}-1\right)=1\\
\iff & \displaystyle \gcd~\left(p^\ell+1,p^{2n}-1\right)=2\\
\iff & \displaystyle \frac{2n}{\gcd~(\ell,2n)}~\text{is odd}.
\end{align*}
But since $\ell$ and $n$ are odd, $\displaystyle \frac{2n}{\gcd~(\ell,2n)}$ is never odd and we are done.
\end{proof}
In view of Theorem~\ref{l3}, it remains to check perfect $(-1)$-nonlinearity of the map $\displaystyle x^{\frac{p^\ell +1}{2}}$ only under the first three conditions of Theorem~\ref{l2} which essentially make it a permutation of $\mathbb F_{p^n}$. Notice that the first three conditions of Theorem~\ref{l2} have a common property that $\ell$ is even. Thus, it makes sense to assume that $\ell$ is even and prove the following theorem that gives  necessary and sufficient conditions on $\ell$ and $n$ for which the power map $x^{\frac{p^\ell +1}{2}}$ is perfect $(-1)$-nonlinear over $\mathbb{F}_{p^n}$  (compare with~\cite[Theorem 8]{RS20}, which also investigates the map).
\begin{thm}
\label{mainp}
The power map $\displaystyle x^{\frac{p^\ell +1}{2}}$ is perfect $(-1)$-nonlinear over $\mathbb{F}_{p^n}$ if and only if any one of the following conditions holds\textup{:}
 \begin{enumerate}
  \item[$(1)$] $\ell = 0$\textup{;}
   \item[$(2)$] $\ell$ even and $n$ odd\textup{;}
   \item[$(3)$] $\ell$ even and $n$ even together with $t_2 \geq t_1+1$, where $n = 2^{t_1}u$ and $\ell = 2^{t_2}v$ such that $2 \nmid u,v$.
 \end{enumerate}
\end{thm}
\begin{proof} From Theorem ~\ref{l2} and Theorem ~\ref{l3}, it is clear that we need to check the perfect $(-1)$-nonlinearity of the map $\displaystyle x^{\frac{p^\ell +1}{2}}$ only when $\ell$ is even. The case $\ell =0$ is trivial. Suppose $\ell \neq 0$. Since $\ell$ is even, $\displaystyle \frac{p^{\ell}+1}{2}$ is odd.  Now by the similar arguments as in the proof of Theorem~\ref{l3} based on Lemma~\ref{equiv-def}, Theorem~\ref{maint} and Lemma~\ref{DPP} , we arrive at the following
\begin{equation*}
\begin{split}
\displaystyle x^{\frac{p^\ell +1}{2}}~\text{is PcN over}~\mathbb F_{p^n}~\text{if and only if}~ \displaystyle \frac{2n}{\gcd~(\ell,2n)}~\text{is odd}. 
\end{split}
\end{equation*}
It is easy to see that  $\displaystyle \frac{2n}{\gcd~(\ell,2n)}$ is odd if and only if one of the latter two conditions of the statement of the theorem is true and thus, we are done. 
\end{proof}
\begin{rmk}
 Observe that Theorem~\textup{\ref{mainp}}  gives a simpler proof of~\cite[Theorem 5]{ZM}, which, in turn, provides a simpler proof of a conjecture of Bartoli and Timpanella~\cite[Conjecture 4.7]{BT}, already settled in \cite{ZM}.
\end{rmk}

\section{Power maps with Low $c=-1$ differential Uniformity}\label{S3}

Due to their wide range of applications in symmetric key cryptography,  functions with low differential uniformity are very important objects. In this section, we give some classes of power maps (monomials) with low $c$-differential uniformity for $c=-1$. We first recall a useful lemma~\cite{SL} related to the Dickson polynomial of the first kind, which is more general than Lemma~\ref{DPP} (see~\cite{WN}).

\begin{lem}\cite[Proposition 41]{SL}
 \label{Dgcd}
Let $a\in\mathbb{F}_{p^n}^*$, and let $D_d(x, a)$ be the Dickson polynomial of the first kind. Then $D_d(x, a)$ is an $m$-to-$1$ function over  $\mathbb{F}_{p^n}$ if and only if $\displaystyle \gcd\left(d, p^{2n} -1\right) = m$.
\end{lem}

Now we shall prove the following theorem that gives $(-1)$-differential uniformity of the map $\displaystyle x^{\frac{p^\ell+1}{2}}$ over $\mathbb{F}_{p^n}$ under certain restrictions. Two of us found the $(-1)$-uniformity of this map in its generality in~\cite{RS20}, but with much more effort, so we thought that the following simpler approach in the next theorem is worth including here, albeit the result being weaker.

\begin{thm}
\label{uniformity}
Let $\displaystyle x^{\frac{p^{\ell}+1}{2}}$ be a power map from $\mathbb{F}_{p^n}$ to itself and $\gcd (\ell, 2n)=1$, $p$ an odd prime.  If $p\equiv 1 \pmod 4$, or $p\equiv 3 \pmod 8$, then the $(-1)$-differential uniformity of $\displaystyle x^{\frac{p^\ell+1}{2}}$ over $\mathbb{F}_{p^n}$ is $ \displaystyle \frac{p+1}{2}$. 
\end{thm}
\begin{proof} 
Since $\gcd (\ell, 2n)=1$, $\ell$ is odd. Thus, $p\equiv 1 \pmod 4$ implies that $p^{\ell}+1 \equiv 2 \pmod 4 $, i.e, $\displaystyle \frac{p^\ell+1}{2}$ is odd (we will only show the first claim as the second is rather similar: we, however, use that if $p\equiv 3 \pmod 8$ implies that $p^{\ell}+1 \equiv 4 \pmod 8 $, that is, $\displaystyle \frac{p^\ell+1}{4}$ is odd). Now we will show that for all $a,b \in \mathbb{F}_{p^n}$, the following equation
\begin{equation}\label{eq3.1}
 (x+a)^{\frac{p^{\ell}+1}{2}}+x^{\frac{p^{\ell}+1}{2}} = b
\end{equation}
has at most $\displaystyle \frac{p+1}{2}$ solutions in $\mathbb{F}_{p^n}$. We first consider the case when $a=0$. In this case,  Equation~\eqref{eq3.1} can have at most $\gcd \left(\frac{p^{\ell}+1}{2}, p^n-1\right)$ roots. By Lemma~\ref{gcd}, if $n$ is odd, then $\displaystyle \gcd \left(p^{\ell}+1, p^n-1\right)=2$ and if $n$ is even, then $\displaystyle \gcd \left(p^\ell+1, p^n-1\right)= p+1$. Therefore, $\displaystyle \gcd \left(\frac{p^{\ell}+1}{2}, p^n-1\right)=1$ for $n$ odd and $\displaystyle \gcd \left(\frac{p^{\ell}+1}{2}, p^n-1\right)=\displaystyle \frac{p+1}{2}$ for $n$ even. Thus, for $a=0$, Equation~\eqref{eq3.1} can have at most $\displaystyle \frac{p+1}{2}$ solutions. We can be more precise:  for $a=0$, then Equation~\eqref{eq3.1} has one solution for $n$ odd and exactly $\frac{p+1}{2}$ solutions for $n$ even for some $b$, and we argue that  below. Let $\alpha$ be a primitive root in $\F_{p^n}$ and $\frac{b}{2}=\alpha^k$, for some~$k$. With $x=\alpha^y$, Equation~\eqref{eq3.1}  becomes $\alpha^{\frac{p^\ell+1}{2} y}=\alpha^k$.  We are reduced to the equation 
\begin{equation}
\label{eq3.1_1}
\frac{p^\ell+1}{2} y\equiv k\pmod{p^n-1}.
\end{equation}
If $\displaystyle \gcd \left(\frac{p^{\ell}+1}{2}, p^n-1\right)=m\in\left\{1,{\frac{p+1}{2}} \right\}$,   then Equation~\eqref{eq3.1_1} has solutions if and only if $m\,|\,k$, and under that assumption, using elementary number theory, there are exactly $m$  solutions $y$ for Equation~\eqref{eq3.1_1}, and they are $y_0,y_0+\frac{p^n-1}{m},y_0+2\frac{p^n-1}{m},\ldots, y_0+(m-1)\frac{p^n-1}{m}$, where $y_0=\frac{k}{m}\left(\frac{p^{\ell}+1}{2m} \right)^{-1}\pmod {\frac{p^n-1}{m}}$, thus inferring our claim (those $b$ for which we have the claim are of the form $b=2\alpha^k$, with $k\equiv 0\pmod m$).

 In the case of $a \neq 0$, we can take $a=1$ in~\eqref{eq3.1}. After relabelling, it is equivalent to find the maximum number of solutions of the equation
\begin{equation}
 (x+1)^{\frac{p^{\ell}+1}{2}}+(x-1)^{\frac{p^{\ell}+1}{2}}= b'
\end{equation}
in $\mathbb{F}_{p^n}$, where $b'\in \mathbb{F}_{p^n}$.
By Theorem~\ref{maint}, the above equation can be re-written as
\begin{equation}\label{eq3.3}
 D_{\frac{p^{\ell}+1}{2}}\left(x, \frac{1}{4}\right)= b'.
\end{equation}
Now, by Lemma~\ref{gcd}, we have $\gcd\left(p^{\ell}+1, p^{2n}-1\right)= p+1$ and hence, $\displaystyle \gcd \left(\frac{p^{\ell}+1}{2}, p^{2n}-1\right)= \displaystyle \frac{p+1}{2}$. Therefore, by Lemma~\ref{Dgcd}, Equation~\eqref{eq3.3} can have at most $\displaystyle \frac{p+1}{2}$ roots, however, with the bound being attained, otherwise $D_{\frac{p^{\ell}+1}{2}}\left(x, \frac{1}{4}\right)$ would not be $m$-to-$1$. This completes the proof.
\end{proof} 
The following are  immediate corollaries to Theorem~\ref{uniformity}.
\begin{cor}
Let $F(x)=x^{\frac{5^{\ell}+1}{2}}$ be a power function on $\mathbb{F}_{5^n}$, $G(x)=x^{\frac{13^{\ell}+1}{2}}$ on $\mathbb{F}_{13^n}$, and $\mbox{gcd}~(\ell, 2n)=1$. Then for $c=-1$, the $c$-differential uniformity of the function $F$ is $3$ and  the one of $G$ is~$7$. 
\end{cor}

\begin{cor}
Let $\displaystyle F(x)=\displaystyle x^{\frac{3^{\ell}+1}{2}}$ be a power function on $\mathbb{F}_{3^n}$, $G(x)=\displaystyle x^{\frac{11^{\ell}+1}{2}}$ on $\mathbb{F}_{{11}^n}$, and  $\gcd~(\ell,2n)=1$. Then for $c=-1$, $F$ is an APcN function \textup{(}see also~\textup{\cite[Thm. 10]{CDU}}\textup{)}, and the $(-1)$-differential uniformity of $G$ is $6$. 
\end{cor}

\section{P$c$N power functions over $\mathbb{F}_{p^5}$ with $c=-1$} 
\label{S4}

In this section, first we shall prove four propositions, which will be useful in the sequel.
\begin{prop}
\label{equv}
Let $c \in \mathbb{F}_p^*$ then the $c$-differential uniformity of the power functions $x^d$  and $x^{dp^j}, j\in \{0,1,\ldots,n-1\}$  over $\mathbb{F}_{p^n}$ is the same.
\end{prop}
\begin{proof}
For $a, b \in \mathbb{F}_{p^n}$, we have
\begin{equation*}
\begin{split}
 (x+a)^d-cx^d = b & \iff x^{p^j} \circ \left((x+a)^d-cx^d\right) = x^{p^j}(b)\\
 &\iff (x+a)^{dp^j} - cx^{dp^j} = e, ~ \mbox{where}~x^{p^j}(b)=e \in  \mathbb{F}_{p^n}.
 \end{split}
\end{equation*}
Since $x^{p^j}$ is a permutation, if $b$ runs over $\mathbb{F}_{p^n}$ then so does $e$. This completes the proof.
\end{proof}

\begin{prop}
\label{inv}
Let $c=\pm1$ and $\mbox{gcd}~(d, p^n-1)=1$, then the $c$-differential uniformity of the power functions $x^d$ and $x^{d^{-1}}$ over $\mathbb{F}_{p^n}$ is the same, where $d^{-1}$ is the inverse of $d$ modulo $p^n-1$.
\end{prop}
\begin{proof}
For any $a,b\in \mathbb{F}_{p^n}$, we have
\begin{equation*}
\begin{split}
 (x+a)^d-cx^d = b &\iff (x+a)^d = (cx^d+b)\\
& \iff x+a = (cx^d+b)^{d^{-1}}\\
& \iff a = (cx^d+b)^{d^{-1}}-x  \\
& \iff a = (y+b)^{d^{-1}} - \frac{y^{d^{-1}}}{c^{d^{-1}}},~\mbox{where}~ y=cx^d  \\
& \iff a = (y+b)^{d^{-1}}-cy^{d^{-1}}  \\
 \end{split}
\end{equation*}
Therefore, for $c=\pm1$, the $c$-differential uniformity of $x^d$ and $x^{d^{-1}}$ over $\mathbb{F}_{p^n}$ is the same.
\end{proof}

\begin{prop}
 Let $p$ be an odd prime and $d'=p^4+(p-2)p^2+(p-1)p+1$. Then  for $c=-1$, the map $x^{d'}$ is PcN over $\mathbb{F}_{p^5}$. 
\end{prop}
\begin{proof}
From Theorem~\ref{mainp}, we know that for $c=-1$, $\displaystyle x^{\frac{p^2+1}{2}}$  is PcN over $\mathbb{F}_{p^5}$. Now since $\displaystyle \gcd~\left(\frac{p^2+1}{2},p^5-1\right)=1$, its multiplicative inverse modulo $p^5-1$ exists and is equal to $p^4+(p-2)p^2+(p-1)p+1$. Therefore, by Proposition~\ref{inv}, $x^{d'}$ is a PcN over $\mathbb{F}_{p^5}$ for $c=-1$.
\end{proof}

In view of {Proposition~\ref {equv}}, Proposition~\ref{inv} and Theorem~\ref{mainp}, the following proposition immediately follows from the fact, stated in \cite{ZM},  that over $\mathbb{F}_{p^n}^*$ with $n$ odd, $\displaystyle p\left(\frac{p^n+1}{p+1}\right)$ is the inverse of $\displaystyle \frac{p^{n-1}+1}{2}$.

\begin{prop}
 Let $p$ be an odd prime and $\displaystyle d={\frac{p^5+1}{p+1}}$. Then  for $c=-1$, $x^d$ is PcN over $\mathbb{F}_{p^5}$. 
\end{prop}

As an empirical support for these results, and in search of more PcN power functions for $c=-1$, we performed an exhaustive search of all possible exponents $d$ for which $x^d$ is PcN for $c=-1$ over the finite fields $\mathbb{F}_{3^5}$, $\mathbb{F}_{5^5}$, and $\mathbb{F}_{7^5}$, respectively. The result of this search was that $d$ is of the form 
\begin{equation*}
 p^j \bigg \{ 1, \frac{p^2+1}{2} , p^4+(p-2)p^2+(p-1)p+1,  \frac{p^4+1}{2} , \frac{p^5+1}{p+1} \bigg\}
\end{equation*}
for all $ 0 \leq j \leq 4$, for $p=3,\,5,\,7$, respectively.
Based on this empirical evidence, we propose the following conjecture.
\begin{conj} \label{C1}
 {Let $p$ be an odd prime.
 Then, for $c=-1$, and for all $ 0 \leq j \leq 4$,}
\begin{equation*}
 p^j \bigg \{ 1, \frac{p^2+1}{2} , p^4+(p-2)p^2+(p-1)p+1,  \frac{p^4+1}{2} , \frac{p^5+1}{p+1} \bigg\}
\end{equation*}
are the only values of $d$ for which $x^d$ is PcN on $\mathbb{F}_{p^5}.$
\end{conj}

\section{P$c$N power functions over $\mathbb{F}_{p^7}$ with $c=-1$} \label{S5}

\begin{prop}
Let $p$ be an odd prime and $d_1=(p-1)p^6+p^5+(p-2)p^3+(p-1)p^2+p$. Then for $c=-1$, the map $x^{d_1}$ is a PcN map over $\mathbb{F}_{p^7}$. 
\end{prop}
\begin{proof}
From Theorem~\ref{mainp}, we know that for $c=-1$, $\displaystyle x^{\frac{p^2+1}{2}}$  is PcN map over $\mathbb{F}_{p^7}$. Now since $\displaystyle \gcd~\left(\frac{p^2+1}{2},p^7-1\right)=1$, its multiplicative inverse modulo $p^7-1$ exists and is equal to $(p-1)p^6+p^5+(p-2)p^3+(p-1)p^2+p$. Thus, by Proposition~\ref{inv}, the map $x^{d_1}$ is PcN function over $\mathbb{F}_{p^7}$ for $c=-1$.
\end{proof}

\begin{prop}
Let $p$ be an odd prime and $d_2=(p-2)p^6+(p-2)p^5+(p-1)p^4+p^3+p^2+p$. Then for $c=-1$, the map $x^{d_2}$ is a PcN function over $\mathbb{F}_{p^7}$. 
\end{prop}
\begin{proof}
From Theorem~\ref{mainp}, we know that for $c=-1$, the power function $\displaystyle x^{\frac{p^4+1}{2}}$  is PcN over $\mathbb{F}_{p^7}$. Now since $\displaystyle \gcd~\left(\frac{p^4+1}{2},p^7-1\right)=1$, its multiplicative inverse modulo $p^7-1$ exists and is equal to $(p-2)p^6+(p-2)p^5+(p-1)p^4+p^3+p^2+p$. Therefore. by Proposition~\ref{inv}, the map $x^{d_2}$ is PcN function over $\mathbb{F}_{p^7}$ for $c=-1$.
\end{proof}

In view of {Proposition~\ref {equv}}, Proposition~\ref{inv} and Theorem~\ref{mainp},  the following proposition is a direct consequence of the fact that over $\displaystyle \mathbb{F}_{p^n}^*$ with $n$ odd, $\displaystyle p\left(\frac{p^n+1}{p+1}\right)$ is the inverse of $\displaystyle \frac{p^{n-1}+1}{2}$.

\begin{prop}
\label{Z2}
 Let $p$ be an odd prime and $d_3=  {\frac{p^7+1}{p+1}}$. Then for $c=-1$, the power function $x^{d_3}$ is PcN over $\mathbb{F}_{p^7}$. 
\end{prop}

As an empirical support for these results, and in search of more PcN power functions for $c=-1$, we performed an exhaustive search of all possible exponents $d$ for which $x^d$ is PcN for $c=-1$ over the finite fields $\mathbb{F}_{3^7}$, $\mathbb{F}_{5^7}$, and $\mathbb{F}_{7^7}$, respectively. The result of this search was that $d$ is of the form \begin{equation*}
\begin{split}
& p^j \bigg \{1, \frac{p^2+1}{2}, ((p-1)p^6+p^5+(p-2)p^3+(p-1)p^2+p),   \frac{p^4+1}{2}, \\
&\qquad \frac{p^6+1}{2}, (p-2)p^6+(p-2)p^5+(p-1)p^4+p^3+p^2+p, 
 \frac{p^7+1}{p+1}  \bigg \} 
\end{split} 
\end{equation*} 
for all $0 \leq j \leq 6$, for $p=3,\,5,\,7$, respectively.

\begin{conj} \label{C2}
{Let $p$ be an odd prime. 
Then for $c=-1$ and for all $0 \leq j \leq 6$,}
\begin{equation*}
\begin{split}
& p^j \bigg \{1, \frac{p^2+1}{2}, ((p-1)p^6+p^5+(p-2)p^3+(p-1)p^2+p),   \frac{p^4+1}{2}, \\
&\qquad \frac{p^6+1}{2}, (p-2)p^6+(p-2)p^5+(p-1)p^4+p^3+p^2+p, 
 \frac{p^7+1}{p+1}  \bigg \} 
\end{split} 
\end{equation*} 
are the only values of $d$ for which $x^d$ is PcN over $\mathbb{F}_{p^7}.$
\end{conj}

\begin{rmk} 
The pattern in  \cite[Conjecture 5.3]{BT}, Conjecture~\textup{\ref{C1}} and Conjecture~\textup{\ref{C2}} appears to suggest that over a finite field $\mathbb{F}_{p^n}$, where $n$ is odd, the positive integers in the following set 
$$
\left\{ p^j \left\{1, \frac{p^2+1}{2},   \frac{p^4+1}{2}, \ldots, \frac{p^{n-1}+1}{2} \right\}\right\}_{ j=0,1,2,\ldots, r-1}
$$ 
and their multiplicative inverse modulo $(p^n-1)$ are the only possible exponents $d$ for which the power function $x^d$ is PcN for $c=-1$. However, this is not true in general and the smallest example is $d=29$ over the finite field~$\mathbb{F}_{3^9}$. Therefore, the question about the exponents $d$,  for which the power functions $x^d$ are PcN over finite field $\mathbb{F}_{p^n}$, where $n$ odd, is not clear, even conjecturally. 
\end{rmk}

\section{Perturbations of P$c$N and other functions}
\label{S6}

After linear functions and power functions, linearized polynomials are another special class containing permutation polynomials. The following proposition gives a necessary and sufficient condition for a linearized polynomial to be perfect $c$-nonlinear, similar to~\cite[Proposition 2.4]{CS97}.

\begin{prop}\label{LPP}
Let $c\neq 1$. A linearized polynomial $L$ is perfect $c$-nonlinear over $\mathbb{F}_{p^n}$ if and only if $L$ is a permutation polynomial if and only if its only root in $\mathbb{F}_{p^n}$ is zero.
\end{prop}
\begin{proof}
 Recall that a linearized polynomial $L(x)$ over finite field $\mathbb{F}_{p^n}$ is a polynomial of the form $\sum_{i=0}^{n-1}a_ix^{p^i}$. Now consider the difference function
\begin{equation*}
 \begin{split}
   _cD_aL (x)&= L(x+a)-cL(x) \\
   &= \sum_{i=0}^{n-1}a_i(x+a)^{p^i} -c \cdot  \sum_{i=0}^{n-1}a_ix^{p^i}\\
   &= (1-c) \cdot \sum_{i=0}^{n-1}a_ix^{p^i} + \sum_{i=0}^{n-1}a_i a^{p^i}.
 \end{split}
\end{equation*}
Now, if the only root of $L(x)$ in $\mathbb{F}_{p^n}$ is zero, then $L(x)$ is a permutation polynomial. Now since $c\neq1$, the difference function $_cD_a$ being an affine linearized polynomial is also a permutation polynomial and hence $L(x)$ is perfect $c$-nonlinear.
\end{proof}

\begin{cor}
Let $c\neq 1$.  The binomial $F(x)=x^{p^j}-ax^{p^i}$, $0\leq i<j$, is a perfect c-nonlinear function over $\mathbb{F}_{p^n}$ if and only if $a$ is not a $(p^{j-i}-1)$-st power in $\mathbb{F}_{p^n}$ and $c\neq 1$. 
\end{cor}
\begin{proof} 
If $a$ is not a $(p^{j-i}-1)$-th power in $\mathbb{F}_{p^n}$ then the only root of $F(x)$ in $\mathbb{F}_{p^n}$ is $0$ and hence $F(x)$ is a linearized permutation polynomial and the result follows from Proposition~\ref{LPP}.
\end{proof}

It is not a simple matter to characterize when a perturbation of a function with some specific property is preserved.  We can, however, characterize when the sum of a PcN and an arbitrary $p$-ary  function is also PcN (for $1\neq c\in \f_p$), thus extending  in some direction the previous corollary. 
\begin{thm}
\label{thm:sumPcN}
Let $1\neq c\in\F_p$ be fixed, $p$ odd. Let $F$ be a perfect $c$-nonlinear function, and $f$ be an arbitrary $p$-ary function, both on $\f_{p^n}$. Then,
$F+\gamma f$ is perfect $c$-nonlinear if and only if for any $\lambda\in\f_{p^n}$ with $\Tr(\gamma\lambda)=\beta\in\f_p^*$, the following is true 
\[
\cW_{R_a}(-\lambda,\beta)=\displaystyle \sum_{y\in\f_{p^n}} \zeta^{\Tr(\beta R_a(y)+\lambda y)}=0,
\]
where $\zeta$ is a $p$-root of unity, $R_a=H_a\circ G^{-1}$, ${_c}D_af(x)=\Tr(H_a(x))$ ($H_a$ is non-unique) and $G^{-1}$ is the compositional inverse of $G={_c}D_a F$.
\end{thm}
\begin{proof}
Certainly, $F+\gamma f$ is PcN if and only if
\begin{align*}
&F(x+a)+\gamma f(x+a) - cF(x)-c\gamma f(x)\\
&=F(x+a)-cF(x)+\gamma (f(x+a)-cf(x)) \\
&={_c}D_a F(x)+\gamma\cdot {_c}D_a f(x)
\end{align*}
is a permutation polynomial. 

We now write ${_c}D_af(x)=\Tr(H_a(x))$, for some (non-unique) function $H_a$ on $\F_{p^n}$ (since $c\in\F_p$,  if $f$ is $p$-ary, then ${_c}D_af$ is $p$-ary, and such $H_a$ does exists). We then use~\cite[Theorem 2]{CK09}, which states that if $G$ is a permutation and $H$ is arbitrary, then $G(x)+\gamma\Tr(H(x))$ is a permutation polynomial if and only if for any $\lambda\in\F_{p^n}$ with $\Tr(\gamma\lambda)=\beta\in\f_p^*$ then
$
\displaystyle \sum_{y\in\f_{p^n}} \zeta^{\Tr(\beta R(y)+\lambda y)}=0,
$
where $R=H\circ G^{-1}$. Our theorem is shown.
\end{proof}

What can we say about a Boolean perturbation of a non-permutation? 
Let $F=L+\gamma f$. From~\cite[Proposition 3]{CK09}, we know that if $F$ is a PP then the linearized polynomial $L$ on $\F_{p^n}$ must be a permutation or a $p$-to-1  map (surely, in general a linearized polynomial can have a kernel with dimension higher than~$1$, but the quoted result shows that if $L$ is a $p^s$-to-1 ($s>1$) function, then $F$ cannot be a PP). We denote by $\im(L)=\{L(x)\,|\, x\in \F_{p^n}\}$, the image of the map $L$.
If $L$ is a permutation polynomial, then Theorem~\ref{thm:sumPcN} applies, so we consider the case of a $p$-to-1 linearized polynomial.
\begin{thm}
Let $1\neq c\in \F_{p}$, $L$ be a $p$-to-$1$   linearized polynomial on $\F_{p^n}$ and $f$ an arbitrary $p$-ary function, and let $F=L+\gamma f$ be a permutation polynomial. 
Then  $F=L+\gamma f$  is perfect $c$-nonlinear  if and only if  both of the following conditions are satisfied for all $a\in\F_{p^n}^*$:
\begin{itemize}
\item[$(i)$] $\gamma \not\in {\rm Im} (L)$;
\item[$(ii)$] ${_c}D_a f(x+\epsilon)-{_c}D_a f(x) \ne 0$, for all $x\in\F_{p^n}$, $\epsilon\in \Ker (L)^*$.
\end{itemize}
\end{thm}
\begin{proof}
Let $a\in\F_{p^n}^*$ and $1\neq c \in \F_p$. Notice that 
\begin{align*}
 {_c}D_a L(x)& = L(x+a)-cL(x)\\
 &=(1-c)L(x)+L(a)\\
 &=L((1-c)x+a).
\end{align*}
Therefore, $\im({_c}D_a L)\subseteq \im (L)$. Further, as we know, $F$ is perfect $c$-nonlinear if and only if 
\begin{align*}
{_c}D_aF(x)= {_c}D_a L(x)+\gamma\, {_c}D_a f(x)&=L((1-c)x+a)+\gamma (f(x+a)-cf(x))
\end{align*}
is a permutation polynomial. 
  
We now slightly modify the proof of~\cite[Theorem 4]{CK09}, since, as it is, it cannot be applied directly for our case. Further, observe that  
\[
{_c}D_aF(x)=
\begin{cases}
L((1-c)x+a)&\text{ if } f(x+a)-cf(x)=0;\\
L((1-c)x+a)+\gamma d &\text{ if } f(x+a)-cf(x)=d\in\F_{p}^*.
\end{cases}
\]
If $\gamma\in \im(L)$, then $\gamma=L(\alpha)$, $\alpha\in\F_{p^n}$, and for $d\in\F_{p}^*$, $\gamma d=dL(\alpha)=L(d\,\alpha)$. Therefore, the image set of ${_c}D_aF$ is contained in the image set of $L$. Consequently, ${_c}D_aF$ cannot be a permutation as $L$ is a $p$-to-$1$ function. Thus, we can assume that $\gamma\not\in\im(L)$. For any $\epsilon \in \Ker(L)^*$, we have
\begin{align*}
 &{_c}D_aF(x+\epsilon)-{_c}D_aF(x) \\
 &= L((1-c)(x+\epsilon)+a)- L((1-c)x+a)+ \gamma ({_c}D_a f(x+\epsilon)-{_c}D_a f(x)) \\
 &= L((1-c)\epsilon)+ \gamma ({_c}D_a f(x+\epsilon)-{_c}D_a f(x))\\
 &= \gamma ({_c}D_a f(x+\epsilon)-{_c}D_a f(x))
\end{align*}
Thus, if ${_c}D_aF$ is a permutation, then ${_c}D_a f(x+\epsilon)-{_c}D_a f(x)$ has to be non-zero for all $x \in \F_{p^n}$ and $\epsilon \in \Ker(L)^*$.

Conversely, we assume that $(i)$ and $(ii)$ hold. Let $y,z\in\F_{p^n}$ such that ${_c}D_aF(y)={_c}D_aF(z)$. Thus 
\begin{align*}
 {_c}D_aF(y)-{_c}D_aF(z) &= 0\\
 L((1-c)(y-z))+\gamma({_c}D_af(y)-{_c}D_af(z))&=0.
\end{align*}
Let $y-z= \epsilon$, then the above equation reduces to 
\[
  (1-c)L(\epsilon)+\gamma\left({_c}D_a f(z+\epsilon)-{_c}D_a f(z)\right)=0.
\]
If $\epsilon \in \Ker(L)$, then by condition $(ii)$, $\epsilon = 0$, forcing $y=z$. If $\epsilon\notin \Ker(L)$, then ${_c}D_a f(y)-{_c}D_a f(z)= \tilde d\in\F_p^*$, so $0=(1-c)L(y-z)+\gamma \tilde d$, contradicting the fact that $\gamma\not\in \im(L)$.
\end{proof}

We shall use below some results of~\cite[Theorem 3]{CK08} and~\cite[Corollary 1(i)]{CK09}.
\begin{thm}
\label{thm:CKcombined}
Let $p$ be a prime number, $\beta, \gamma \in \F_{p^n}$ and $H\in\F_{p^n}[x]$. Then the polynomial
\[
F(x)=x+\gamma\Tr(H(x^p-\gamma^{p-1}x)+\beta x)
\]
is a permutation polynomial if and only if $\Tr(\beta\gamma)\neq -1$.
\end{thm}
(Surely, if $p=2$, the trace condition is $\Tr(\beta\gamma)=0$.)
We are now ready to show the next result, where we construct a class of (linearized) polynomials that are PcN for every $c\neq 1$, in all characteristics.
\begin{prop}
Let $p$ be a prime number, $\alpha,\gamma\in\F_{p^n}$. Then $F(x)=x+\gamma\Tr(x^p-\alpha x)$ is   PcN for all $c\neq 1$ if and only if $\Tr(\gamma(1-\alpha))\neq -1$.
\end{prop}
\begin{proof}
The $c$-differential of $F$ at $a$ is now
\allowdisplaybreaks
\begin{align*}
{_c}D_a F(x)&=F(x+a)-cF(x)\\
&=x+a+\gamma\Tr\left(x^p+a^p-\alpha x-\alpha a \right)-cx-\gamma c\Tr(x^p-\alpha x)\\
&=(1-c)x+(1-c)\gamma\Tr(x^p-\alpha x) +a+\gamma\Tr(a^p-\alpha a).
\end{align*}
Thus, $F$ is PcN if and only if  $(1-c)x+(1-c)\gamma\Tr(x^p-\alpha x) +a+\gamma\Tr(a^p-\alpha a)$ is PP for all $a$, which is equivalent to 
 $(1-c)x+(1-c)\gamma\Tr(x^p-\alpha x)$ being a PP, and further, $x+\gamma\Tr(x^p-\alpha x)$ being a PP. Now, we re-write the previous function as 
 $x+\Tr\left(x^p-\gamma^{p-1} x+(\gamma^{p-1}-\alpha)x\right)$. Using Theorem~\ref{thm:CKcombined} with $\beta=\gamma^{p-1}-\alpha$, we see that the last claim will hold if and only if $\Tr\left(\gamma\left(\gamma^{p-1}-\alpha\right)\right)=\Tr\left(\gamma^p\right)-\Tr\left(\gamma\alpha\right)  =\Tr(\gamma(1-\alpha))\neq -1$.
\end{proof}

We saw that some modifications of PcN functions preserve their perfect $c$-nonlinearity. It surely makes sense to ask whether the $c$-differential uniformity is preserved through affine, extended affine or CCZ-equivalence~\cite{CCZ}. Given a function $F$, we call the set $\{\beta_{F,c}\,|\, c\in \F_{p^n}\}$, the {\em differential spectrum} of $F$.
We ask here the question of whether that the differential uniformity spectra is preserved under the A-equivalence, EA-equivalence, or CCZ-equivalence. Our guess was that it is not preserved by EA, nor CCZ-equivalence, and an easy computation via SageMath confirmed it: while $x^3$ has $c$-differential spectrum $[1,2,3]$, the EA-equivalent function $x^3+x^4$ has $c$-differential spectrum $[1,2,3,4]$, both on $\F_{2^4}$.

 It is not difficult to show that the differential spectrum is invariant under the (restricted to input) affine-equivalence (A-equivalence) (recall that $F,F'$ on $\F_{2^n}$ are restricted to input A-equivalent if $F'(x)=F\circ \cL (x)$, where $\cL$ is an affine permutation on $\F_{2^n}$), and we provide the argument next. The equation $F'(x+a)-c F'(x)=b$ is equivalent to  $(F\circ\cL)(x+a)-c(F\circ\cL)(x)=b$, that is $F(\cL(x)+\cL(a))-cF(\cL(x))=b$. Setting $\cL(x)=y,\cL(a)=\alpha$, the previous equation becomes $F(y+\alpha)-cF(y)=b$. Surely, any solution of  $F'(x+a)-c F'(x)=b$ is in one-to-one correspondence to a solution of $F(y+\alpha)-cF(y)=b$, since $\cL$ is invertible.
 
Since the CCZ-equivalence is more general than EA-equivalence, we shall concentrate on it.  Recall that two $(n,m)$-functions $F,F'$ from $\F_{p^n}$ to $\F_{p^m}$ are CCZ-equivalent if and only if their graphs $G_F=\{(x,F(x))\,|\, x\in\F_{p^n}\}$, $G_{F'}=\{(x,F'(x))\,|\, x\in\F_{p^n}\}$ are affine equivalent, that is, there exists an affine permutation $\cA$ on $\F_{p^n}\times \F_{p^m}$ such that $\cA(G_F)=G_{F'}$.

As in \cite{CCZ}, we use the  identification of  the elements in $\F_{p^n}$ with the elements in $\F_p^n$, and denote by $x$ both an element in $\F_{p^n}$ and the corresponding element in $\F_p^n$.
We first decompose the affine permutation $\cA$ as an affine block-matrix,
$\cA{\bf u}=\begin{pmatrix}
\cA_{11} &\cA_{12}\\
\cA_{21} &\cA_{22}
\end{pmatrix} {\bf u}+\colvec{c}{d}$, for an input vector $\bf u$, where $\cA_{11},\cA_{21}$,  $\cA_{12},\cA_{22}$ are $n \times n$ matrices with entries in $\F_{p}$, and $\colvec{c}{d}$ is a column vector in $\F_{p^{2n}}$ (just a reminder to the reader that EA-equivalence means that $\cA_{12}=0$ and (full-fledged) $A$-equivalence means that $\cA_{12}=\cA_{21}=0$).
Fix $c\in\F_{p^n}$, and let the $c$-differential system be written as
$y-x=a, F(y)-cF(x)=b$.
 Applying the affine permutation $\cA$ to $\colvec{a}{b}$ we get
 \allowdisplaybreaks
 \begin{align*}
 \begin{pmatrix}
\cA_{11} &\cA_{12}\\
\cA_{21} &\cA_{22}
\end{pmatrix}
\colvec{a}{b}&=
\begin{pmatrix}
\cA_{11} &\cA_{12}\\
\cA_{21} &\cA_{22}
\end{pmatrix}
\colvec{y-x}{F(y)-cF(x)}\\
&=
\begin{pmatrix}
\cA_{11} &\cA_{12}\\
\cA_{21} &\cA_{22}
\end{pmatrix}
\colvec{y}{F(y)}-
\begin{pmatrix}
\cA_{11} &\cA_{12}\\
\cA_{21} &\cA_{22}
\end{pmatrix}
\colvec{x}{cF(x)}\\
&=\colvec{y'}{F'(y')}-\begin{pmatrix}
\cA_{11} &c\cdot\cA_{12}\\
\cA_{21} &c\cdot\cA_{22}
\end{pmatrix}
\colvec{x}{F(x)}.
 \end{align*}
We see that it is not obvious how the second term can be transformed into a pair $\colvec{x'}{c^*F'(x')}$ of the graph $G_{F'}$, unless $F,F'$ are also CCZ-equivalent also via an affine transformation whose linear part is a constant multiple of $\begin{pmatrix}
\cA_{11} &c\cdot\cA_{12}\\
\cA_{21} &c\cdot\cA_{22}
\end{pmatrix}$.
We summarize this discussion in the next theorem.
\begin{thm}
Let $F,F'$ be CCZ-equivalent via an affine transformation $\cA=\begin{pmatrix}
\cA_{11} &\cA_{12}\\
\cA_{21} &\cA_{22}
\end{pmatrix}$ 
and also via 
$\begin{pmatrix}
\frac{1}{c^*}\cdot \cA_{11} &\frac{c}{c^*} \cdot\cA_{12}\\
\frac{1}{c^*}\cdot \cA_{21} &\frac{c}{c^*} \cdot\cA_{22}
\end{pmatrix}$. Then the $c$-differential uniformity of $F$ is the same as the $c^*$-differential uniformity of $F'$.
\end{thm}


 With the above discussion, we see that the $c$-differential uniformity may change under EA or CCZ-equivalence. Keeping that in mind, we now switch directions a bit and ask whether we can perturb some APcN functions, via a linear/linearized map, thereby obtaining a PcN function.
This is in line with the long standing open question on whether some of the known PN or APN functions can be transformed into PN or APN permutation functions by perturbing them via some linear mapping.  We will only treat here the Gold case, $F(x)=x^{p^k+1}$. From~\cite{CDU} we know that $F$ is PcN only for $c=1$ (under $\frac{n}{\gcd(n,k)}$ odd), when $p>2$, and it is never PcN for $c\neq 1$. The case of $p=2$ was treated in~\cite{RS20}.

\begin{thm}
\label{thm:perturb}
 Let $k\geq 1,n\geq 2$ be  integers, $p$   prime, $c\neq 1$ in~$\F_{p^n}$. The following are true:
\begin{enumerate}
\item[$(i)$]
If $G_1(x)=x^{p^k+1}+\gamma\Tr(x)$ is perfect $c$-nonlinear for $\gamma\in\F_{p^n}^*$, then 
\[
\gamma\not\in\left \{-\frac{a^{p^k+1}}{  \Tr\left(\frac{a}{1-c}\right) (1-c)^2} \,{\bigg |}\, a\in\F_{p^n}^*, \Tr\left(\frac{a}{1-c}\right) \neq 0
\right\}.
\]
\item[$(ii)$] 
The function $G_2(x)=x^{p^k+1}+\gamma x^{p^k}$ is never PcN, regardless of the value of~$\gamma\in\F_{p^n}^*$.
\end{enumerate}
\end{thm}
 \begin{proof}
(i) We first perturb $F$ in the following way $G_1(x)=F(x)+\gamma \Tr(x)$, $\gamma\neq 0$, and attempt to find some condition on $\gamma$ such that $G_1$ can potentially be PcN. We look at the $c$-differential equation of $G_1$, namely
\[ 
(1-c) x^{p^k+1}+a\, x^{p^k}+a^{p^k} x+a^{p^k+1}+\gamma(1-c)\Tr(x)+\gamma \Tr(a)=b,
\] 
that is,
\begin{align*}
&  x^{p^k+1}+\frac{a}{1-c} x^{p^k}+\frac{a^{p^k}}{1-c} x+ \gamma \Tr(x)=\frac{b-\gamma\Tr(a)-a^{p^k+1}}{1-c}.
\end{align*}
By relabeling (since the free term is linear in $b$), it will be sufficient to investigate the equation
\[
 x^{p^k+1}+\frac{a}{1-c} x^{p^k}+\frac{a^{p^k}}{1-c} x+ \gamma \Tr(x)=b.
\]
We argue now that in many instances the equation has more than one solution. We let  $b=0$. Surely, $x=0$ is one such solution. We write (for $a\neq0$)
\begin{align*}
x^{p^k}\left( x+\frac{a}{1-c}\right) + \frac{a^{p^k}}{1-c} \left(x+ \frac{\gamma(1-c)}{a^{p^k}} \Tr(x)\right)=0.
\end{align*}
Now, $x=-\frac{a}{1-c}\neq 0$ is another solution if $ \frac{\gamma(1-c)}{a^{p^k}} \Tr\left(-\frac{a}{1-c}\right)=\frac{a}{1-c}$, or, equivalently, $ \Tr\left(\frac{a}{1-c}\right)=-\frac{a^{p^k+1}}{\gamma (1-c)^2}$. We obviously need $\frac{a^{p^k+1}}{\gamma (1-c)^2}\in\F_p^*$, for some $a$, which is equivalent to  the first claim.

(ii) Next, we perturb $F$ as $G_2(x)=F(x)+\gamma\, x^{p^k}$, $\gamma\neq 0$.
As before, the $c$-differential equation of $G_2$ is then
\[
(1-c) x^{p^k+1}+a\, x^{p^k}+a^{p^k} x+a^{p^k+1}+\gamma((1-c)x^{p^k}+a^{p^k})=b,
\]
or, by relabeling $\frac{b-a^{p^k+1}-\gamma a^{p^k}}{1-c}\mapsto b$
\[
x^{p^k+1}+\frac{a+\gamma(1-c)}{1-c}\, x^{p^k}+\frac{a^{p^k} }{1-c} x=b.
\]
If $b=0$, then $x=0$ is a solution. Assuming $b=0,x\neq 0,a\neq0$, factoring out $x$, and using $y=\frac{1}{x}$, we get 
\[
y^{p^k}+\frac{a+\gamma(1-c)}{a^{p^k}}\, y+\frac{1-c}{a^{p^k}}=0.
\]
It is easy to show that   taking $a=\gamma(c-1)$, then $y=\left(\frac{c-1}{a^{p^k}} \right)^{p^{-k}}$ (which always exists, since $\gcd(p^k, p^n-1)=1$) is a solution of the above equation, and hence $x=\left(\frac{a^{p^k}}{c-1} \right)^{p^{-k}}$ is a solution of the original equation in~$x$. Hence ${_c}D_a G_2$ is not a permutation, and therefore, $G_2$ is not PcN, for $c\neq 1$. 
\end{proof}

Surely, the question is whether $G_1(x)=x^{2^k+1}+\gamma \Tr(x)$ is ever PcN over $\F_{2^n}$.
We quickly took some small examples of $\F_{2^n}$,
$2\leq n\leq 4$, determined by the primitive polynomials $x^2+x+1$, $x^3+x+1$, $x^4+x+1$ over $\F_2$, all with  some primitive root $\alpha$. We then
  checked that  $G_1(x)=x^{2^k+1}+\gamma \Tr(x)$ is never PcN  on $\F_{2^n}$, for $2\leq k<n\leq 4$. If $k=n$, we can get PcN functions. For the considered cases,  if $(k,n)=(2,2)$, $G_1$ is PcN when $(c,\gamma)=(0,1), (\alpha,1),( \alpha^2,1)$; if $(k,n)=(3,3)$, $G_1$ is PcN when $(c,\gamma)=(c,\alpha), (c,\alpha^2), (c,\alpha^4)$, since the function $G_1$ becomes a linearized polynomial (via $x^{2^n+1}=x^2$ on $\F_{2^n}$).  We do not have other examples for small dimensions. The computation was done via SageMath.

\section{Further comments}
\label{S7}

In this paper, in the first part, we used  Dickson polynomials techniques to show some results (some were independently shown recently). We also found that recently published necessary conditions, which give a relationship between the difference function of a monomial  and the Dickson polynomial of first kind (odd characteristic), are also sufficient (Theorem~\ref{maint}). Next, we give several classes of PcN and functions with low $(-1)$-differential uniformity, and we propose two conjectures based upon some computational data. We also obtain a class of polynomials that are PcN for all $c\neq 1$, in every characteristic.
Further, we discuss the affine, extended affine and CCZ-equivalence as it relates to $c$-differential uniformity.
We then concentrate on perturbation of a PcN function to also be perfect $c$-nonlinear and give necessary and sufficient conditions in some cases. We also show that in some instances such perturbations do not produce PcN functions.
Surely, it would be very interesting to find other perturbations, linear or not, that may decrease the $c$-differential uniformity. For example, one can computationally check that a ``switching'' technique~\cite{BKL09,YKP06} produces some PcN functions as well: the functions $x^3+\gamma\Tr(x^3),x^9+\gamma\Tr(x^3)$ are PcN on $\F_{2^3}$ for $c=0,\gamma=\alpha,\alpha^2,\alpha^4$.

 \vskip.3cm
\noindent
{\bf Acknowledgments}. The authors  would like to express their sincere appreciation for the reviewers’ careful reading, beneficial comments and suggestions, and to the editors for the prompt handling of our paper. The research of Sartaj Ul Hasan is partially supported by Start-up Research Grant SRG/2019/000295 from the Science and Engineering Research Board, Government of India.


\begin{thebibliography}{99} 
\bibitem{BT} D. Bartoli, M. Timpanella, {\it On a generalization of planar functions}, J. Algebr. Comb. 52 (2020),187--213.

 \bibitem{Bori} N. Borisov, M. Chew, R. Johnson, D. Wagner, {\it Multiplicative
differentials}. In: J. Daemen and V. Rijmen (eds.) Proceedings of Fast Software Encryption - FSE 2002. Lecture Notes in Comput. Sci., Springer, Berlin, Heidelberg, vol. 2365 (2002), 17--33. 

\bibitem{Bo90}
N. Bourbaki, Elements of Mathematics, Algebra II (translated by P. M. Cohn and J. Howie), Springer, Berlin, 1990.

\bibitem{BKL09}
L. Budaghyan, C. Carlet, G. Leander, {\em Constructing new APN functions from known ones}, Finite Fields Appl. 15 (2009), 150--159.

\bibitem{CCZ}
C. Carlet, P. Charpin, V. Zinoviev, {\em Codes, bent functions and permutations suitable for DES-like cryptosystems}, Des. Codes Cryptogr. 15 (1998), 125--156.

\bibitem{CK08}
P. Charpin, G. Kyureghyan, {\em On a class of permutation polynomials over $\f_{2^n}$} In: Golomb S.W., Parker M.G., Pott A., Winterhof A. (eds) Proceedings of Sequences and Their Applications - SETA 2008. Lecture Notes in Comput. Sci., Springer, Berlin, Heidelberg, vol 5203 (2008), 368--376.

\bibitem{CK09}
P. Charpin, G. Kyureghyan, {\em When does $G(x)+\gamma Tr(H(x))$ permute $\f_{p^n}$}, Finite Fields Appl. 15(5) (2009), 615--632.

\bibitem{CS97}
R. S. Coulter, R. W. Matthews,
{\em Planar functions and planes of Lenz-Barlotti class} II, Des. Codes Cryptogr. 10 (1997), 167--184.

\bibitem{YKP06} 
Y. Edel, G. Kyureghyan, A. Pott, {\em A new APN funciton which is not equivalent to a power mapping}, IEEE Trans. Inform. Theory 52:2 (2006), 744--747.
 
\bibitem{CDU} P. Ellingsen, P. Felke, C. Riera, P. St\u anic\u a, A. Tkachenko, {\it C-differentials, multiplicative uniformity and (almost) perfect $c$-nonlinearity}, IEEE Trans. Inform. Theory 66:9 (2020), 5781--5789.

\bibitem{WN} W. N{\"o}bauer, {\it {\"U}ber eine Klasse von Permutationspolynomen und die dadurch dargestellten Gruppen}, J. Reine Angew. Math. 231 (1968), 215--219.

\bibitem{SL} S. Mesnager, L. Qu, {\it On two-to-one mappings over finite fields}, IEEE Trans. Inf. Theory 65:12 (2019), 7884--7895.

\bibitem{RS20}
C. Riera, P. St\u anic\u a, {\em Some c-(almost) perfect nonlinear functions},
\url{https://arxiv.org/abs/2004.02245}.

\bibitem{Li} X. Xu, C. Li, X. Zeng, T. Helleseth,{ \it Constructions of complete permutation polynomials}, Des. Codes Cryptogr. 86 (2018), 2869--2892.  

\bibitem{ZM} H. Yan, S. Mesnager, and Z. Zhou {\it Power Functions over Finite Fields with Low $c$-Differential Uniformity}, \url{https://arxiv.org/abs/2003.13019}.
 
\end{thebibliography}
 \end{document}